\documentclass[notitlepage,11pt,reqno]{amsart}
\usepackage[foot]{amsaddr}
\usepackage[numbers,sort]{natbib}

\usepackage{amssymb,nicefrac,bm,upgreek,mathtools,verbatim}
\usepackage[final]{hyperref}
\usepackage[mathscr]{eucal}
\usepackage{dsfont}

\usepackage{color}
\definecolor{dmagenta}{rgb}{.4,.1,.5}
\definecolor{dblue}{rgb}{.0,.0,.5}
\definecolor{mblue}{rgb}{.0,.0,.8}
\definecolor{ddblue}{rgb}{.0,.0,.4}
\definecolor{dred}{rgb}{.6,.0,.0}
\definecolor{dgreen}{rgb}{.0,.5,.0}
\definecolor{Eeom}{rgb}{.0,.0,.5}

\usepackage[normalem]{ulem}

\usepackage[margin=1in]{geometry}
\allowdisplaybreaks

\newtheorem{lemma}{Lemma}[section]
\newtheorem{theorem}{Theorem}[section]

\theoremstyle{definition}
\newtheorem{definition}{Definition}[section]
\newtheorem{assumption}{Assumption}[section]

\newtheorem{example}{Example}[section]
\newtheorem{remark}{Remark}[section]

\numberwithin{equation}{section}

\hypersetup{
  colorlinks=true,
  citecolor=dblue,
  linkcolor=dblue,
  frenchlinks=false,
  pdfborder={0 0 0},
  naturalnames=false,
  hypertexnames=false,
  breaklinks}
\usepackage[capitalize,nameinlink]{cleveref}

\crefname{section}{Section}{Sections}
\crefname{subsection}{Subsection}{Subsections}
\crefname{condition}{Condition}{Conditions}
\crefname{hypothesis}{Hypothesis}{Conditions}
\crefname{assumption}{Assumption}{Assumptions}
\crefname{lemma}{Lemma}{Lemmas}
\crefname{claim}{Claim}{Claims}

\Crefname{figure}{Figure}{Figures}

\crefformat{equation}{\textup{#2(#1)#3}}
\crefrangeformat{equation}{\textup{#3(#1)#4--#5(#2)#6}}
\crefmultiformat{equation}{\textup{#2(#1)#3}}{ and \textup{#2(#1)#3}}
{, \textup{#2(#1)#3}}{, and \textup{#2(#1)#3}}
\crefrangemultiformat{equation}{\textup{#3(#1)#4--#5(#2)#6}}%
{ and \textup{#3(#1)#4--#5(#2)#6}}{, \textup{#3(#1)#4--#5(#2)#6}}%
{, and \textup{#3(#1)#4--#5(#2)#6}}

\Crefformat{equation}{#2Equation~\textup{(#1)}#3}
\Crefrangeformat{equation}{Equations~\textup{#3(#1)#4--#5(#2)#6}}
\Crefmultiformat{equation}{Equations~\textup{#2(#1)#3}}{ and \textup{#2(#1)#3}}
{, \textup{#2(#1)#3}}{, and \textup{#2(#1)#3}}
\Crefrangemultiformat{equation}{Equations~\textup{#3(#1)#4--#5(#2)#6}}%
{ and \textup{#3(#1)#4--#5(#2)#6}}{, \textup{#3(#1)#4--#5(#2)#6}}%
{, and \textup{#3(#1)#4--#5(#2)#6}}

\crefdefaultlabelformat{#2\textup{#1}#3}

\makeatletter
\DeclareRobustCommand\widecheck[1]{{\mathpalette\@widecheck{#1}}}
\def\@widecheck#1#2{%
    \setbox\z@\hbox{\m@th$#1#2$}%
    \setbox\tw@\hbox{\m@th$#1%
       \widehat{%
          \vrule\@width\z@\@height\ht\z@
          \vrule\@height\z@\@width\wd\z@}$}%
    \dp\tw@-\ht\z@
    \@tempdima\ht\z@ \advance\@tempdima2\ht\tw@ \divide\@tempdima\thr@@
    \setbox\tw@\hbox{%
       \raise\@tempdima\hbox{\scalebox{1}[-1]{\lower\@tempdima\box
\tw@}}}%
    {\ooalign{\box\tw@ \cr \box\z@}}}

\def\subsection{\@startsection{subsection}{0}%
\z@{\linespacing\@plus\linespacing}{\linespacing}%
{\bf}}

\makeatother

\DeclareMathOperator{\Exp}{\mathbb{E}} 
\DeclareMathOperator{\Prob}{\mathbb{P}} 
\newcommand{\D}{\mathrm{d}}          
\newcommand{\RR}{\mathbb{R}}         
\newcommand{\Rd}{{\mathbb{R}^d}}       
\newcommand{\Ind}{\mathds{1}}            

\newcommand{\cA}{\mathcal{A}}
\newcommand{\sB}{\mathscr{B}}    

\newcommand{\cC}{\mathcal{C}}     
\newcommand{\cD}{\mathrm{D}}     

\newcommand{\sM}{\mathscr{M}}     
\newcommand{\cO}{\mathcal{O}}

\newcommand{\cS}{\mathcal{S}}
\newcommand{\cT}{\mathcal{T}}

\newcommand{\abs}[1]{\lvert#1\rvert}
\newcommand{\norm}[1]{\lVert#1\rVert}

\providecommand{\pro}[1]{(#1_t)_{t \geq 0}}

\providecommand{\seq}[1]{(#1_n)_{n\in \mathbb{N}}}

\DeclareMathOperator{\diam}{diam}

\newcommand{\ex}{\mathbb{E}}

\DeclareMathOperator{\dist}{dist}

\DeclareMathOperator{\Deg}{\mathrm{deg}}

\DeclareMathOperator{\Psidel}{\Psi(-\Delta)}

\begin{document}

\title[Semilinear nonlocal operators with exterior condition]%
{\sc \textbf{Existence and non-existence results for a class of  semilinear nonlocal operators with exterior condition}}

\author[Anup Biswas]{Anup Biswas}

\address{Anup Biswas \\
Department of Mathematics, Indian Institute of Science Education and Research, Dr. Homi Bhabha Road,
Pune 411008, India, anup@iiserpune.ac.in}

\date{}


\begin{abstract}
We consider a class of semilinear nonlocal problems with vanishing exterior condition
and establish a Ambrosetti-Prodi type phenomenon 
when the nonlinear term satisfies certain conditions.
Our technique makes use of the probabilistic tools and heat kernel estimates.
\end{abstract}

\maketitle

\section{Introduction}
In a seminal work \cite{AP72} Ambrosetti and Prodi consider the problem
\begin{equation}\label{AP}
-\Delta u + f(u)= h(x)\quad \text{in}\; \cD, \quad \text{and}\quad u=0\quad \text{on}\; \partial\cD,\tag{AP}
\end{equation}
for a bounded $\cC^{2, \alpha}$ domain $\cD$  and study existence of solutions for the above problem. The authors have shown that
provided $f$ is strictly convex with $f(0)=0$ and 
$$0< \lim_{s\to-\infty}\, f'(s)< \lambda_1<\lim_{s\to\infty} f'(s)<\lambda_2,$$
where $\lambda_1, \lambda_2$ are the first two eigenvalues of $-\Delta$, there exists a $\cC^1$ manifold $\sM_1$ in $\cC^\alpha(\bar\cD)$ which splits
the space $\cC^\alpha(\bar\cD)$ into two open sets $\sM_0$ and $\sM_2$ with the following property: \eqref{AP} has no solution for $h\in\sM_0$, exactly one
solution for $h\in\sM_1$ and exactly two solutions for $h\in\sM_2$. Following this fundamental observation, much work has been done in the direction of
relaxing the conditions or generalizing it to non-linear partial equations or systems. In
\cite{BP75} Berger and Podolak propose a useful reformulation of the above problem as follows.
\begin{equation}\label{BP}
-\Delta u = f(u) + \rho\,\Phi_1 + h(x)\quad \text{in}\; \cD, \quad \text{and}\quad u=0\quad \text{on}\; \partial\cD,\tag{BP}
\end{equation}
where $\Phi_1$ is the principal eigenfunction of the Laplace operator in $\cD$. Under suitable conditions, it is shown in \cite{BP75} that for a real $\rho^*=\rho^*(h)$,
\eqref{BP} has no solution for $\rho>\rho^*$, it has exactly one solution for $\rho=\rho^*$ and exactly two solutions for $\rho<\rho^*$.
For further developments  on Ambrosetti-Prodi type problems we refer to \cite{AH79, D78, FS84, MRZ, KW75, S14} and references therein. 
There are also some recent works on Ambrosetti-Prodi problems involving fractional Laplacian operators, see \cite{BL18, P17}.

The goal of this article is to generalize the above results to a wider class of operators such as $\Psi(-\Delta)$. By $-\Psi(-\Delta)$ we denote the generator of a 
subordinate Brownian motion where the subordinator having Laplace exponent given by $\Psi$. See Example~\ref{Eg2.1} below for some interesting examples of $\Psi(-\Delta)$. More precisely, given a bounded $\cC^{1,1}$ domain $\cD$ we consider the problem
\begin{equation*}
\left\{\begin{array}{ll}
\Psidel u = f(x, u) + \rho\, \Phi_1 + h(x) \quad \text{in}\, \; \cD, \\
\quad \quad \;\;\; u = 0\, \qquad\qquad\qquad\qquad \quad\, \text{in}\;\; \cD^c,
\end{array} \right.
\end{equation*}
where $f, h$ are given continuous functions and $f$ satisfies Ambrosetti-Prodi type conditions (see Assumption [AP] below). One of our main results (Theorem~\ref{T-AP})
can be informally stated as follows.

\vspace{.1in}

\noindent{\it There exists a real $\rho^*$ such that the above problem does not have any solution for $\rho>\rho^*$, at least one 
solution for $\rho=\rho^*$ and at least two solutions for $\rho<\rho^*$.}

The central idea of the proof remains the same as in \cite{BL18, DFS, S14} where one has to construct a minimal solution for certain values of $\rho$, and 
find bounds $\norm{u}_{L^\infty(\cD)}$ and  then use a degree theory argument to get to the conclusion. Some key tools required for this
methodology to work are (1) {\it refined maximum principle} (see Theorem~\ref{T2.2} below), (2) boundary behaviour of the solutions
and (3) Hopf's lemma. Recently, a version of refined maximum principle is obtained in \cite{BL17b} whereas the boundary behaviour of the solution has been obtained by \cite{KKLL}. As a substitute to the Hopf's lemma we use the heat kernel estimates from \cite{BGR14, CKS}. With these tools in hand we employ  more technical arguments,
compare to the existing literature, to obtain our results.

\section{Setting and statement of main result}

\subsection{Subordinate Brownian motion}
A Bernstein function is a non-negative completely monotone function, i.e., an element of the set
$$
\mathcal B = \left\{f \in C^\infty((0,\infty)): \, f \geq 0 \;\; \mbox{and} \:\; (-1)^n\frac{d^n f}{dx^n} \leq 0,
\; \mbox{for all $n \in \mathbb N$}\right\}.
$$
In particular, Bernstein functions are increasing and concave. We will make use below of the subset
$$
{\mathcal B}_0 = \left\{f \in \mathcal B: \, \lim_{u\downarrow 0} f(u) = 0 \right\}.
$$
Let $\mathcal M$ be the set of Borel measures $\mu$ on $\RR \setminus \{0\}$ with the property that
$$
\mu((-\infty,0)) = 0 \quad \mbox{and} \quad \int_{\RR\setminus\{0\}} (y \wedge 1) \mu(dy) < \infty.
$$
Notice that, in particular, $\int_{\RR\setminus\{0\}} (y^2 \wedge 1) \mu(dy) < \infty$ holds, thus $\mu$ is a L\'evy
measure supported on the positive semi-axis. It is well-known then that every Bernstein function $\Psi \in
{\mathcal B}_0$ can be represented in the form
\begin{equation}
\Psi(u) = bu + \int_{(0,\infty)} (1 - e^{-yu}) \mu(\D{y})
\end{equation}
with $b \geq 0$, moreover, the map $[0,\infty) \times \mathcal M \ni (b,\mu) \mapsto \Psi \in {\mathcal B}_0$ is
bijective. $\Psi$ is said to be a complete Bernstein function (see \cite[Chapter~6]{SSV}) if there exists a Bernstein function $\widetilde\Psi$
such that
$$
\Psi(u)= u^2 \mathcal{L}(\widetilde\Psi)(u), \quad u>0\,,
$$
where $\mathcal{L}$ stands for the Laplace transformation. It is known that every complete Bernstein function is also
a Bernstein function. Also, for a complete Bernstein function the L\'evy measure $\mu(\D{y})$ has a completely monotone
density with respect to the Lebesgue measure. The class of complete Bernstein functions is large, including important
cases such as (i) $u^{\alpha/2}, \, \alpha\in(0, 2]$; (ii) $(u+m^{2/\alpha})^{\alpha/2}-m, \, m\geq 0, \alpha
\in (0, 2)$; (iii)$u^{\alpha/2} + u^{\beta/2}, \, 0 < \beta < \alpha \in(0, 2]$; (iv) $\log(1+u^{\alpha/2})$, $\alpha
\in (0,2]$, (v) $u^{\alpha/2}(\log(1+u))^{\beta/2}$, $\alpha \in (0,2)$, $\beta \in (0, 2-\alpha)$, (vi) $u^{\alpha/2}
(\log(1+u))^{-\beta/2}$, $\alpha \in (0,2]$, $\beta \in [0,\alpha)$. On the other hand, the Bernstein function
$1 - e^{-u}$ is not a complete Bernstein function. For a detailed discussion of Bernstein functions we refer to the
monograph \cite{SSV}.

Bernstein functions are closely related to subordinators, and we will use this relationship below. Recall that a
one-dimensional L\'evy process $\pro S$ on a probability space $(\Omega_S, {\mathcal F}_S, \mathbb P_S)$ is called a
subordinator whenever it satisfies $S_s \leq S_t$ for $s \leq t$, $\mathbb P_S$-almost surely.
A basic fact is that the Laplace transform of a subordinator is given by a Bernstein function, i.e.,
\begin{equation}
\label{lapla}
\ex_{\mathbb P_S} [e^{-uS_t}] = e^{-t\Psi(u)}, \quad t \geq 0,
\end{equation}
holds, where $\Psi \in {\mathcal B}_0$. In particular, there is a bijection between the set of subordinators on a given
probability space and Bernstein functions with vanishing right limits at zero; to emphasize this, we will occasionally
write $\pro {S^\Psi}$ for the unique subordinator associated with Bernstein function $\Psi$. Corresponding to the
examples of Bernstein functions above, the related processes are (i) $\alpha/2$-stable subordinator, (ii) relativistic
$\alpha/2$-stable subordinator, (iii) sums of independent subordinators of different indeces, (iv) geometric $\alpha/2$-stable
subordinators (specifically, the Gamma-subordinator for $\alpha = 2$), etc. The non-complete Bernstein function mentioned
above describes the Poisson subordinator.

Let $\pro B$ be $\Rd$-valued a Brownian motion on Wiener space $(\Omega_W,{\mathcal F}_W, \mathbb P_W)$, running twice
as fast as standard $d$-dimensional Brownian motion, and let $\pro {S^\Psi}$ be an independent subordinator. The random
process
$$
\Omega_W \times \Omega_S \ni (\omega_1,\omega_2) \mapsto B_{S_t(\omega_2)}(\omega_1) \in \Rd
$$
is called subordinate Brownian motion under $\pro {S^\Psi}$. For simplicity, we will denote a subordinate Brownian motion
by $\pro X$, its probability measure for the process starting at $x \in \Rd$ by $\mathbb P^x$, and expectation with respect
to this measure by $\ex^x$. Note that the characteristic exponent of $\pro X$ is given by $\Psi(\abs{x}^2)$. It is also known that the
L\'evy measure of $X$ has a density $y\mapsto j(\abs{y})$ where $j:(0, \infty)\to (0, \infty)$ is given by
\begin{equation}\label{E2.3}
j(r) = \int_0^\infty (4\pi t)^{-d/2} e^{-\frac{r^2}{4t}} \, \mu(\D{t}),
\end{equation}
and 
\begin{equation}\label{E2.4}
\Psi(\abs{z}^2)= \int_{\Rd\setminus\{0\}} (1-\cos(y\cdot z)) j(\abs{y}) \, \D{y}.
\end{equation}
We would be interested in the following class of Bernstein functions.
\begin{definition}
\label{WLSC}
The function $\Psi$ is said to satisfy a
\begin{enumerate}
\item[(i)]
weak lower scaling (WLSC) property with parameters ${\underline\mu} > 0$, $\underline{c} \in(0, 1]$ and
$\underline{\theta}\geq 0$, if
$$
\Psi(\gamma u) \;\geq\; \underline{c}\, \gamma^{\underline\mu} \Psi(u), \quad u>\underline{\theta}, \; \gamma\geq 1.
$$
\item[(ii)]
weak upper scaling (WUSC) property with parameters $\bar\mu > 0$,
$\bar{c} \in[1, \infty)$ and $\bar{\theta}\geq 0$, if
$$
\Psi(\gamma u) \;\leq\; \bar{c}\, \gamma^{\bar{\mu}} \Psi(u), \quad u>\bar{\theta}, \; \gamma\geq 1.
$$
\end{enumerate}
\end{definition}

\begin{example}\label{Eg2.1}
Some important examples of $\Psi$ satisfying WLSC and WUSC include the following
cases with the given parameters, respectively:
\begin{itemize}
\item[(i)]
$\Psi(u)=u^{\alpha/2}, \, \alpha\in(0, 2]$, with ${\underline\mu} = \frac{\alpha}{2}$, $\underline{\theta}=0$,
and $\bar\mu = \frac{\alpha}{2}$, $\bar{\theta}=0$.
\item[(ii)]
$\Psi(u)=(u+m^{2/\alpha})^{\alpha/2}-m$, $m> 0$, $\alpha\in (0, 2)$, with ${\underline\mu} = \frac{\alpha}{2}$,
$\underline{\theta}=0$ and $\bar\mu = 1$, $\bar{\theta}=0$.
\item[(iii)]
$\Psi(u)=u^{\alpha/2} + u^{\beta/2}, \, \alpha, \beta \in(0, 2]$, with ${\underline\mu} = \frac{\alpha}{2}
\wedge \frac{\beta}{2}$, $\underline{\theta}=0$ and $\bar\mu = \frac{\alpha}{2} \vee \frac{\beta}{2}$,
$\bar{\theta}=0$.
\item[(iv)]
$\Psi(u)=u^{\alpha/2}(\log(1+u))^{-\beta/2}$, $\alpha \in (0,2]$, $\beta \in [0,\alpha)$ with
${\underline\mu}=\frac{\alpha-\beta}{2}$, $\underline{\theta}=0$ and $\bar\mu=\frac{\alpha}{2}$, $\bar{\theta}=0$.
\item[(v)]
$\Psi(u)=u^{\alpha/2}(\log(1+u))^{\beta/2}$, $\alpha \in (0,2)$, $\beta \in (0, 2-\alpha)$, with
${\underline\mu}=\frac{\alpha}{2}$, $\underline{\theta}=0$ and $\bar\mu=\frac{\alpha+\beta}{2}$, $\bar{\theta}=0$.
\end{itemize}
\end{example}
The following condition will be imposed on $\pro X$.
\begin{assumption}\label{A2.1}
$\Psi$ satisfies both WLSC and WUSC properties with respect to some parameters $(\underline{\mu}, \underline{c}, \underline{\theta})$ and 
$(\bar{\mu}, \bar{c}, \bar{\theta})$, respectively. Moreover, for some positive constant $\varrho$ we have 
\begin{equation}\label{A2.1A}
j(r+1)\geq \varrho\, j(r),\quad \text{for all}\; \; r\geq 1,
\end{equation}
 where $j$ is given by
\eqref{E2.3}.
\end{assumption}
It is obvious that $\bar{\mu}\geq \underline{\mu}$. If $\Psi$ is complete Bernstein and satisfies for some $\alpha\in(0, 1)$ that $\Psi(r)\asymp r^{\alpha}\ell(r)$, as $r\to\infty$, for some locally bounded and slowly varying function $\ell$, then \eqref{A2.1A} holds \cite[Theorem 13.3.5]{KSV}. 
Many results of this article would be valid without Assumption~\ref{A2.1}. However, to establish compactness of certain operators 
(see Theorem~\ref{T2.1} or Lemma~\ref{L3.7} below)
we use some estimates from \cite{KKLL} which uses Assumption~\ref{A2.1}.

For our analysis we also require the renewal function $V$ of the properly normalized ascending ladder-height process of $X^{(1)}_t$, where $X^{(1)}_t$ denotes the 
first coordinate of $X_t$. The ladder-height process is a subordinator with Laplace exponent 
$$\tilde\Psi(\xi)=\exp\left\{\frac{1}{\pi}\int_0^\infty \frac{\log \Psi(\xi\zeta)}{1+\zeta^2}\, \D{\zeta}\right\}, \quad \xi\geq 0\,,$$
and $V(x)$ is its potential measure of the half-line $(-\infty, x)$. The Laplace transform of $V$ is given by 
$$\int_0^\infty V(x) e^{-\xi x}\, \D{x}= \frac{1}{\xi \tilde\Psi(\xi)}, \quad \xi>0.$$
It is also known that $V=0$ for $x\leq 0$,$V$ is continuous and strictly increasing in $(0, \infty)$ and $V(\infty)=\infty$ (see \cite{F74} for more details).
From \cite[Lemma~1.2]{BGR14} it is known that
for some universal constant $C$, dependent only on the dimension $d$, we have
\begin{equation}\label{E2.6}
C^{-1} \,{\Psi(r^{-2})}\leq \frac{1}{V^2(r)}\leq C\, {\Psi(r^{-2})}\quad r>0.
\end{equation}

\subsection{Main results}
Let $\cD$ be a $\cC^{1, 1}$ open bounded set. By $\uptau$ we denote the exit time of $\pro X$ from $\cD$.
Given a function $U\in\cC(\bar\cD)$ called potential, the corresponding Feynman-Kac
semigroup is given by
\begin{equation}\label{FKsemi}
T^{\cD, U}_t f(x)= \Exp^x\left[e^{-\int_0^t U(X_s)\, \D{s}} f(X_t)\Ind_{\{t<\uptau\}}\right], \quad t>0,
\; x\in\cD, \; f\in L^2(\cD)\,.
\end{equation}
It is shown in \cite[Lem~3.1]{BL17a} that
$T^{\cD, U}_t$, $t > 0$, is a Hilbert-Schmidt operator on $L^2(\cD)$ with continuous integral kernel in
$(0, \infty)\times\cD\times\cD$. Moreover, every operator $T_t$ has the same purely discrete spectrum, independent
of $t$, whose lowest eigenvalue is the principal eigenvalue $\lambda^*$ having multiplicity one, and the corresponding
principal eigenfunction $\Phi \in L^2(\cD)$ is strictly positive in $\cD$. Since the boundary of $\cD$ is regular by \cite[proof of Lemma 2.9]{BGR15}
 we also have from \cite[Lem.~3.1]{BL17a} that $\Phi\in\cC_0(\cD)$, where $\cC_0(\cD)$ denotes the class of continuous functions on $\Rd$ vanishing in $\cD^c$.
 Since $\Phi$
is an eigenfunction in semigroup sense, we have for all $t>0$ that
\begin{equation}\label{E2.7}
e^{-\lambda^* t} \Phi(x) = T^{\cD, U}_t \Psi(x) = \Exp^x\left[e^{-\int_0^t U(X_s)\, \D{s}} \Phi(X_t)\Ind_{\{t<\uptau\}}\right],
\quad x\in\cD.
\end{equation}
Moreover, $\lambda$ in \eqref{E2.7} is an eigenvalue of the operator $\Psidel + U$ with Dirichlet exterior condition. By $\lambda^*_U$ we denote the principal
eigenvalue corresponding to the potential $U$ and $\lambda^*=\lambda^*_0$. Let $\Phi_1\in\cC_0(\cD)$ be the positive eigenfunction corresponding to the eigenvalue
$\lambda^*$. We normalize $\Phi_1$ to satisfy $\norm{\Phi_1}_\infty=1$.
In this paper we are interested in the existence and multiplicity
of solutions of
\begin{equation}
\label{E-AP}
\left\{\begin{array}{ll}
\Psidel u = f(x, u) + \rho\, \Phi_1 + h(x) \quad \text{in}\, \; \cD, \\
\quad \quad \;\;\; u = 0\, \qquad\qquad\qquad\qquad \quad\, \text{in}\;\; \cD^c,
\tag{$P_\rho $}
\end{array} \right.
\end{equation}
where $h\in\cC(\bar\cD)$ and $f$ is continuous function satisfying some appropriate condition. In what follows by a solution of 
\begin{equation}\label{E2.8}
\left\{\begin{array}{ll}
\Psidel u  = g  \quad \text{in}\, \; \cD, \\
\quad \quad \;\;\; u = 0\, \quad \, \text{in}\;\; \cD^c,
\end{array} \right.
\end{equation}
for $g\in\cC(\bar\cD)$ we mean semigroup or potential theoretic solution. More precisely, the solution of \eqref{E2.8} is given by
$$u(x) = \int_{\cD} g(y) G^\cD(x, y)\, \D{y}= \Exp^x\left[\int_0^{\uptau} g(X_s)\, \D{s}\right],$$
where $G^\cD$ denotes the Green function of $(X^D_t)_{t\geq 0}$, the killed process of $X$ upon $D$. From the strong Markov property it is easily seen that
\begin{equation}\label{E2.9}
u(x)= \Exp^x\left[\int_0^{t\wedge \uptau} g(X_s)\, \D{s}\right] + \Exp^x[u(X_{t\wedge\uptau})]\quad t\geq 0.
\end{equation}
It can also be shown that the solution of \eqref{E2.8} is also a viscosity solution of \eqref{E2.8} (see \cite{KKLL}). 

Our first result concerns with the existence of solution.
\begin{theorem}\label{T2.1}
Suppose that Assumption~\ref{A2.1} holds. Let $U, g\in \cC(\bar\cD)$ and $\lambda^*_U>0$. Then there exists a unique $u\in \cC_0(\cD)$ satisfying
\begin{equation}\label{ET2.1A}
\Psidel u + U u = g \;\; \mbox{in}\;\; \cD,  \quad u=0 \;\; \mbox{in}\;\; \cD^c.
\end{equation}
\end{theorem}

We also need the following refined  maximum principle.
\begin{theorem}\label{T2.2}
Let Assumption~\ref{A2.1} hold.
Suppose that $U\in\cC(\bar\cD)$ and $\lambda^*_U>0$. Let $u\in\cC_{\rm b}(\Rd)$ be a viscosity solution of
$\Psidel u + U u=g_1$ and $v\in\cC_{\rm b}(\Rd)$ be a viscosity solution of  $\Psidel v + U v=g_2$ in $\cD$ for some
$g_1, g_2\in\cC(\bar\cD)$ with $g_1\leq g_2$. Furthermore, assume that $u=v=0$ in $\cD^c$. Then we have either  $u< v$ in $\cD$ or $u=v$ in $\Rd$.
\end{theorem}

We impose the following Ambrosetti-Prodi type condition on $f$.
\begin{assumptn}\label{Assump-AP}
Let $f:\bar\cD\times\RR\to\RR$ be such that
\begin{enumerate}
\item[(1)]
both $f(x, u)$ and $D_uf(x, u)$ are continuous in $(x, u)\in \bar\cD\times\RR$;
\medskip
\item[(2)]
there exist $U_1, U_2\in\cC(\bar\cD)$ with $U_1\geq U_2$ such that
\begin{align}
\lambda^*_{U_1}&>0 \quad\text{and}\quad \lambda^*_{U_2}<0\,, \label{AP1}
\\[2mm]
f(x, q) &\geq -U_1(x)q-C\quad \text{for all}\; q\leq 0, \; x\in\bar\cD\,, \label{AP2}
\\[2mm]
f(x, q) &\geq -U_2(x)q-C\quad \text{for all}\; q\geq 0, \; x\in\bar\cD\,, \label{AP3}
\end{align}
\item[(3)]
$f$ has at most linear growth, i.e., there exists a constant $C > 0$ such that
$$
\abs{f(x,q)} \leq C(1+\abs{q}),
$$
for all $(x,q)\in\bar\cD\times\RR$.
\end{enumerate}
\end{assumptn}
In what follows, we assume with no loss of generality that $f(x,0)=0$, otherwise $h$ can be replaced by $h-f(\cdot,0)$.
The condition $U_1\geq U_2$ is imposed for some technical reason. As well known this condition is not required when
$\Psi(r)=r^s$ for $s\in(0, 1] $ (see \cite{BL18} and references therein). It should be observed that due to our Assumption [AP](2)
we have $f(x, q)\geq -U_1(x) q - C$ for $q\in\RR$.

Now we are ready to state our main result on the nonlocal Ambrosetti-Prodi problem.
\begin{theorem}\label{T-AP}
Let Assumption~\ref{A2.1} and [AP] hold. Then there exists $\rho^* = \rho^*(h) \in\RR$ such that for $\rho < \rho^*$ the
Dirichlet problem \eqref{E-AP} has at least two solutions, at least one solution for $\rho = \rho^*$, and no
solution for $\rho > \rho^*$.
\end{theorem}

\section{Proofs}
We prove Theorem~\ref{T2.1}-\ref{T-AP} in this section.
The following result would play a key role in our proofs.
\begin{lemma}\label{L3.1}
Let $u\in\cC_0(\cD)$ be a solution of 
$$\Psidel u =g\quad \text{in}\; \cD,$$
 for some $g\in\cC(\bar\cD)$. Consider $U\in\cC(\bar\cD)$. Then for any $t\geq 0$ we have
\begin{equation}\label{EL3.1A}
\Exp^x\left[e^{\int_0^t U(X_s)\, \D{s}} u(X_t)\Ind_{\{t<\uptau\}}\right]-u(x)=
\Exp^x\left[\int_0^{t\wedge\uptau} e^{\int_0^{s} U(X_p)\, \D{p}}(U(X_s)u(X_s)-g(X_s))\, \D{s}\right], \quad x\in\cD.
\end{equation}
\end{lemma}

\begin{proof}
Define
$$\psi(t) = \Exp^x\left[e^{\int_0^{t\wedge\uptau} U(X_s)\, \D{s}} u(X_{t\wedge\uptau})\right]
=\Exp^x\left[e^{\int_0^t U(X_s)\, \D{s}} u(X_t)\Ind_{\{t<\uptau\}}\right].$$
From \cite[Lemma 3.1]{BL17a} it follows that $\psi$ is continuous in $[0, \infty)$. We fix $t\geq 0$ and consider $h>0$.
Then
\begin{align}\label{EL3.1B}
\psi(t+h)-\psi(t) & = \Exp^x\left[e^{\int_0^{(t+h)\wedge\uptau} U(X_s)\, \D{s}} u(X_{(t+h)\wedge\uptau})\right]
-\Exp^x\left[e^{\int_0^{t\wedge\uptau} U(X_s)\, \D{s}} u(X_{t\wedge\uptau})\right]\nonumber
\\
&= \Exp^x\left[\left(e^{\int_0^{(t+h)\wedge\uptau} U(X_s)\, \D{s}}- e^{\int_0^{t\wedge\uptau} U(X_s)\, \D{s}}\right) u(X_{(t+h)\wedge\uptau})\right]\nonumber
\\
&\ \qquad + \Exp^x\left[e^{\int_0^{t\wedge\uptau} U(X_s)\, \D{s}} \left(u(X_{(t+h)\wedge\uptau})- u(X_{t\wedge\uptau})\right)\right]\nonumber
\\
&= \underbrace{\Exp^x\left[e^{\int_0^{t\wedge\uptau} U(X_s)\, \D{s}} 
\left(e^{\int_{t\wedge\uptau}^{(t+h)\wedge\uptau} U(X_s)\, \D{s}}- 1\right) u(X_{(t+h)\wedge\uptau})\right]}_{A_1(h)}\nonumber
\\
&\ \qquad +  \underbrace{\Exp^x\left[e^{\int_0^{t\wedge\uptau} U(X_s)\, \D{s}} 
\left(\Exp^{X_{t\wedge\uptau}}[u(X_{h\wedge\uptau})]- u(X_{t\wedge\uptau})\right)\right]}_{A_2(h)},
\end{align}
where in the last line we used strong Markov property. Since $u(X_{(t+h)\wedge\uptau})=0$ on $\{t\geq \uptau\}$, it follows that
$$A_1(h)= \Exp^x\left[e^{\int_0^{t} U(X_s)\, \D{s}} 
\left(e^{\int_{t\wedge\uptau}^{(t+h)\wedge\uptau} U(X_s)\, \D{s}}- 1\right) u(X_{(t+h)\wedge\uptau})\Ind_{\{t<\uptau\}}\right],$$
and therefore, applying dominated convergence theorem we obtain
\begin{equation}\label{EL3.1C}
\lim_{h\to 0} \frac{A_1(h)}{h}= \Exp^x\left[e^{\int_0^{t} U(X_s)\, \D{s}} 
 U(X_t) u(X_{t}) \Ind_{\{t<\uptau\}}\right].
\end{equation}
From \eqref{E2.9} we get that
$$\Exp^{X_{t\wedge\uptau}}[u(X_{h\wedge\uptau})]- u(X_{t\wedge\uptau})=-\Exp^{X_{t\wedge\uptau}}\left[\int_0^{h\wedge \uptau} g(X_s)\,\D{s}\right],$$
since both the sides vanishes on the set $\{t\geq \uptau\}$. Thus again applying dominated convergence theorem we find 
\begin{equation}\label{EL3.1D}
\lim_{h\to 0} \frac{A_2(h)}{h}= -\Exp^x\left[e^{\int_0^{t} U(X_s)\, \D{s}} 
 g(X_{t}) \Ind_{\{t<\uptau\}}\right].
\end{equation}
Hence using \eqref{EL3.1B}, \eqref{EL3.1C} and \eqref{EL3.1D} we obtain
$$\psi^\prime_+(t)= \Exp^x\left[e^{\int_0^{t} U(X_s)\, \D{s}} 
 \left(U(X_t) u(X_{t})-g(X_t)\right) \Ind_{\{t<\uptau\}}\right].$$
It also follows from \cite[Lemma~3.1]{BL17a} that $t\mapsto \psi^\prime_+(t)$ is continuous. Hence $\psi$ is in $\cC^1(0,\infty)$ and by fundamental theorem
of calculus we have
\begin{align*}
\psi(t)- u(x)=\psi(t)- \psi(0) &= \int_0^t \Exp^x\left[e^{\int_0^{s} U(X_p)\, \D{p}} 
 \left(U(X_s) u(X_{s})-g(X_s)\right) \Ind_{\{s<\uptau\}}\right]\, \D{s}
\\
&= \Exp^x\left[\int_0^{t\wedge\uptau} e^{\int_0^{s} U(X_p)\, \D{p}}(U(X_s)u(X_s)-g(X_s))\, \D{s}\right].
\end{align*}
This proves \eqref{EL3.1A}.
\end{proof}

Let us now prove Theorem~\ref{T2.1}.
\begin{proof}[Proof of Theorem~\ref{T2.1}]
The main idea in proving \eqref{ET2.1A} is to use Schauder's fixed point theorem. Consider a map
$\cT:\cC_0(\cD)\to\cC_0(\cD)$ defined such that for every $\psi\in\cC_0(\cD)$, $\cT\psi=\varphi$ is the unique 
solution of
\begin{equation}\label{ET2.1B}
\Psidel\varphi =g-U\psi\quad \text{in}\; \cD, \quad \text{and}\quad \varphi=0\quad \text{in}\; \cD^c.
\end{equation}
Denoting $\bar\phi=[\Psi(r^{-2})]^{-\frac{1}{2}}$ and using \cite[Theorem~1.1]{KKLL} we obtain that
\begin{equation}\label{ET2.1C}
\norm{\cT\psi}_{\cC^{\bar\phi}(\cD)}\leq c_1 (\norm{g}_\infty + \norm{U\psi}_\infty),
\end{equation}
for a constant $c_1 = c_1(\cD, d, s)$ where
\begin{equation}\label{phi-norm}
\norm{h}_{\cC^{\bar\phi}(\cD)}=\norm{h}_{L^\infty(\cD)} + \sup_{x\neq y, x, y\in\cD}\frac{\abs{h(x)-h(y)}}{\bar\phi(x-y)}.
\end{equation}
Thus using \eqref{E2.6} and \eqref{ET2.1C} we have
$$\abs{\cT\psi(x)-\cT\psi(y)}\leq c_2 (\norm{g}_\infty + \norm{U\psi}_\infty) V(\abs{x-y}).$$
This implies that $\cT$ is a compact linear operator. It is also easy to see that
$\cT$ is continuous.

In a next step we show that the set
$$
\sB= \big\{\varphi\in\cC_0(\cD)\; :\; \varphi=\mu\cT\varphi \;\; \text{for some}\;\; \mu\in[0, 1]\big\}
$$
is bounded in $\cC_0(\cD)$. For every $\varphi\in\sB$ we have
\begin{equation}\label{ET2.1D}
\Psidel\varphi = \mu g - \mu U\varphi\quad \text{in}\; \cD, \quad \text{and}\quad \varphi=0\quad \text{in}\; \cD^c,
\end{equation}
for some $\mu\in[0,1]$. 
From \eqref{ET2.1D} and Lemma~\ref{L3.1} we see that 
\begin{equation}\label{ET2.1E}
\varphi(x) = \Exp^x\left[e^{-\int_0^t \mu U(X_s)\, \D{s}} \varphi(X_t)\Ind_{\{t<\uptau\}}\right] + \mu \Exp^x\left[\int_0^{t\wedge\uptau} e^{-\int_0^{s} \mu U(X_p)\, \D{p}}g(X_s)\, \D{s}\right], \quad t\geq 0.
\end{equation}
To show boundedness of $\sB$
it suffices to show that for a constant $c_2$, independent of $\mu$, we have
\begin{equation}\label{ET2.1G}
\sup_{x\in\bar\cD}\abs{\varphi(x)}\leq c_2\, \sup_{x\in\bar\cD}\abs{g(x)}.
\end{equation}
Once \eqref{ET2.1G} is established, the existence of a fixed point of $\cT$ follows by Schauder's fixed point theorem.
Since every solution of \eqref{ET2.1B} is a semigroup solution and $\lambda^*>0$, the uniqueness of the solution follows
from \cite[Th.~4.2]{BL17b} and Lemma~\ref{L3.1}. To obtain \eqref{ET2.1G} recall from \cite[Cor.~4.1]{BL17b} that
\begin{equation}\label{ET2.1H}
\lambda^*_{\mu V}=
-\lim_{t\to\infty} \frac{1}{t} \log \Exp^x\left[e^{-\int_0^t \mu U(X_s)\, \D{s}}\Ind_{\{\uptau>t\}}\right], \quad x\in\cD\,.
\end{equation}
Recall that $\lambda^*>0$ is the principal eigenvalue corresponding to the potential $U=0$. Then from the concavity of the map 
$\mu\mapsto \lambda^*_{\mu U}$ (see \cite[Lem.~4.3]{BL17b}) it follows that
$$
\lambda^*_{\mu U}\geq \lambda^*_U \wedge \lambda^*_0 = 2\delta>0.
$$
Hence by using \eqref{ET2.1H} and the continuity of $\mu\mapsto\lambda^*_{\mu U}$, we find constants $c_3>0, \mu_0>1$, such
that for every $\mu\in[0, \mu_0]$ we have
\begin{equation}\label{ET2.1I}
\Exp^x\left[e^{-\int_0^t \mu U(X_s)\, \D{s}}\Ind_{\{\uptau>t\}}\right]\leq c_3 e^{-\delta t},\quad t\geq 0, \; x\in\cD.
\end{equation}
We rewrite \eqref{ET2.1E} as
$$
\varphi(x) = \Exp^x\left[ e^{-\int_0^t \mu U(X_s)\, \D{s}} \varphi(X_t)\Ind_{\{\uptau>t\}}\right] +
\int_0^t T^{\cD, \mu U}_s g (x) \D{s},
$$
where $T^{\cD, \mu U}$ is given by \eqref{FKsemi}.
Letting $t\to\infty$, using \eqref{ET2.1I} and H\"{o}lder inequality, it is easily seen that the first term at the right hand
side of the above vanishes. Again by \eqref{ET2.1I}, we have for $x\in\cD$
$$
\Big|T^{\cD, \mu V}_s g (x)\Big| \leq c_3 \sup_{x\in\bar\cD}\abs{g}\, e^{-\delta s}, \quad s\geq 0.
$$
Thus finally we obtain
$$
\sup_{x\in\bar\cD}\abs{\varphi(x)}\leq\; \frac{c_3}{\delta} \sup_{x\in\bar\cD}\abs{g(x)},
$$
yielding \eqref{ET2.1G}.
\end{proof}

Next we prove the comparison result Theorem~\ref{T2.2}.
\begin{proof}[Proof of Theorem~\ref{T2.2}]
Using Lemma~\ref{L3.1} we see that
\begin{equation*}
u(x) = \Exp^x\left[e^{-\int_0^t U(X_s)\, \D{s}} u(X_t)\Ind_{\{t<\uptau\}}\right] + \Exp^x\left[\int_0^{t\wedge\uptau} e^{-\int_0^{s} U(X_p)\, \D{p}}g_1(X_s)\, \D{s}\right], 
\quad t\geq 0,
\end{equation*}
and,
\begin{equation*}
v(x) = \Exp^x\left[e^{-\int_0^t U(X_s)\, \D{s}} v(X_t)\Ind_{\{t<\uptau\}}\right] + \Exp^x\left[\int_0^{t\wedge\uptau} e^{-\int_0^{s} U(X_p)\, \D{p}}g_2(X_s)\, \D{s}\right], 
\quad t\geq 0.
\end{equation*}
Denoting $w=v-u$ and using the above expressions we obtain
\begin{equation*}
w(x) \geq \Exp^x\left[e^{-\int_0^t U(X_s)\, \D{s}} w(X_t)\Ind_{\{t<\uptau\}}\right], \quad t\geq 0.
\end{equation*}
From \cite[Theorem~4.2]{BL17b} it then follows that either $w>0$ in $\cD$ or $w=0$ in $\Rd$. Hence the proof.
\end{proof}

\begin{remark}
The condition $u=v=0$ in $\cD^c$ in Theorem~\ref{T2.2} is not necessary. In fact, the same argument as above can used to establish comparison principle
provided $u\leq v$ in $\cD^c$.
\end{remark}

The rest of the article is devoted to the proof of Theorem~\ref{T-AP}. The central strategy of the proof can be grouped in following three steps.
\begin{itemize}
\item[(1)] We find a $\rho_1$ such that for every $\rho\leq\rho_1$ there exists a (minimal) solution of \eqref{E-AP}. We do this in Lemma~\ref{L3.2} and ~\ref{L3.3}.
\item[(2)] Next we find $\rho_2>\rho_1$ such that \eqref{E-AP} does not have any solution for $\rho\geq \rho_2$. This is the content of Lemma~\ref{L3.4}, \ref{L3.5}
and ~\ref{L3.6}
\item[(3)] Finally, we proceed along the lines of \cite{DF80} with suitable modifications to find the bifurcation point $\rho^*$.
\end{itemize}

Let us begin by establishing existence of sub/super-solutions, which will be used for constructing a minimal solution.
\begin{lemma}\label{L3.2}
Let Assumptions~\ref{A2.1} and [AP] hold. Then we have the following.
\begin{itemize}
\item[(1)] 
For every $\rho\in\RR$ there exists $\underline{u}\in\cC_0(\cD)$ satisfying $\underline{u}\leq 0$ in $\cD$ 
and
$$
\Psidel \underline{u} = f(x, \underline{u}) + \rho \Phi_1 + h(x) + g(x) \quad \text{in}\;\; \cD,
$$
for some nonpositive $g\in\cC(\bar\cD)$.
\item[(2)] 
There exists ${\rho}_1<0$ such that for every $\rho\leq {\rho}_1$ there exists $\bar{u}\in \cC_0(\cD)$ 
satisfying $\bar{u}\geq 0$ in $\cD$ and
$$
\Psidel \bar{u} = f(x, \bar{u}) + \rho \Phi_1 + h(x) + g(x) \quad \text{in}\;\; \cD,
$$
for some nonnegative $g\in\cC(\bar\cD)$.
\item[(3)] 
We can construct $\underline{u}$ to satisfy $\underline{u}\leq \hat{u}$, for every solution $\hat{u}\in \cC_0(\cD)$ of
$$
\Psidel \hat{u} = f(x, \hat{u}) + \rho \Phi_1 + h(x) + g(x) \quad \text{in}\;\; \cD,
$$
with $g\geq 0$.
\end{itemize}
\end{lemma}

\begin{proof}
Consider $\rho\in\RR$. Let $C_1= 2\sup_{\bar\cD}\abs{h} + 2\abs{\rho} + C $, where $C$ is the same constant as in 
\eqref{AP2}-\eqref{AP3}. Since $\lambda^*_{U_1}>0$ by \eqref{AP1}, it follows from Theorem~\ref{T2.1} that 
there exists a unique $\underline{u}\in\cC_0(\cD)$ satisfying
\begin{equation}\label{EL3.2A}
\Psidel\underline{u} = - U_1 \underline{u}-C_1+h(x)+ \rho\Phi_1\quad \text{in}\; \cD.
\end{equation}
By our choice of $C_1$ we see that
$$
\Psidel\underline{u} + U_1 \underline{u}= -C_1+h(x)+ \rho\Phi_1\leq 0 ,
$$
and hence, by Theorem~\ref{T2.2} we have $\underline{u}\leq 0$ in $\Rd$. Therefore, by making use of \eqref{AP2} 
and choosing $g(x)= -f(x, \underline{u})- U_1 \underline{u}-C_1$
we get that
$$
\Psidel\underline{u} = f(x, \underline{u}) + h(x)+ \rho\Phi_1 + g(x) \quad \text{in}\; \cD, \quad \text{and}\quad 
\underline{u}=0\quad \text{in}\; \cD^c.
$$
This proves part (1).

Now we proceed to establish (2). Due to Assumption [AP] there exists a constant $C_1$ satisfying $f(x, q)\leq C_1 (1+q)$,
for all $(x, q)\in \bar\cD\times[0, \infty)$. We consider the unique function $\bar{u}\in\cC_0(\cD)$ 
satisfying
\begin{equation}\label{EL3.2B}
\Psidel\bar{u}=h^+ + C_1 \quad \text{in}\; \cD.
\end{equation}
Therefore
$$\abs{u(x)} = \left|\Exp^x \left[\int_0^\uptau h^+(X_s) + C_1 \D{s}\right]\right|\leq (\norm{h}_\infty + C_1)\Exp^x[\uptau].$$
Thus by Assumption~\ref{A2.1} and \cite[Theorem~4.6 and Lemma~7.5]{BGR15} we obtain
\begin{equation}\label{EL3.2C}
\abs{u(x)}\leq c_1 V(\delta_\cD(x)), \quad x\in\bar\cD,
\end{equation}
for some constant $c_1$, dependent on $\cD$,
where $\delta_\cD(x)=\dist(x, \cD^c)$. Again
$$\bar{u}(x)\geq C_1 \Exp^x[\uptau]>0\quad \text{for}\; x\in\cD.$$
Let $p^\cD(t, x, y)$ be the transition density of the killed process $X^\cD$ in $\cD$. In fact, one can write
$$p^\cD(t, x, y)= p(t, \abs{x-y})-\Exp^x[p(t-\uptau, \abs{X_{\uptau}-y})\Ind_{\{\uptau<t\}}].$$
Using \cite[Theorem~4.5]{BGR14} (see also \cite{CKS}) we know that for some positive constants $\kappa_1, r$ we have for $x, y\in\cD$
\begin{align}
p^\cD(t, x, y) &\geq \kappa_1 \Prob^x(\uptau>t/2) \Prob^y(\uptau>t/2) p(t\wedge V^2(r), \abs{x-y}), \quad t\geq 0,\label{EL3.2C1}
\\
\Prob^x(\uptau>t) &\geq \kappa_1 \left(\frac{V(\delta_\cD(x))}{\sqrt{t\wedge V(r)}}\wedge 1\right). \label{EL3.2C2}
\end{align}
Now recall that  $\Psidel\Phi_1=\lambda^*\Phi_1$ in $\cD$, and $\Phi_1>0$ in $\cD$. 
Let $\cD_1\Subset\cD$. Fixing $t=2$ and using \eqref{E2.7} we get that
\begin{align*}
\Phi_1(x) &= e^{2\lambda^*} \Exp^x\left[ \Phi_1(X_t)\Ind_{\{2<\uptau\}}\right]
\\
&= e^{2\lambda^*} \int_\cD \Phi_1(y) p^\cD(2, x, y)\, \D{y}
\\
&\geq e^{2\lambda^*} \int_{\cD_1} \Phi_1(y) p^\cD(2, x, y)\, \D{y}
\\
&\geq \kappa_1 e^{2\lambda^*}\, \min_{\cD_1}\Phi_1\, \Prob^x(\uptau>1) \int_{\cD_1} \Prob^y(\uptau>1) p(1\wedge V^2(r), \abs{x-y})\, \D{y}
\\
&\geq \kappa_2\, p(1\wedge V^2(r), 0) \Prob^x(\uptau>1) \int_{\cD_1} \Prob^y(\uptau>1)\, \D{y},
\end{align*}
for some constant $\kappa_2$, where in the fourth inequality we use \eqref{EL3.2C1}. Now using \eqref{EL3.2C2} we can find a constant $\kappa_3>0$ satisfying
$$
\Phi_1(x)\geq \kappa_3 V(\delta_\cD(x)), \quad x\in \cD.
$$
Combining the above with \eqref{EL3.2C} and choosing $-\rho_1>0$ large, we find for every $\rho\leq \rho_1$ that
$$
- \rho\Phi_1(x)\geq C_1c_1 V(\delta_\cD(x))\geq C_1 \bar{u}(x), \quad \text{for}\; x\in\cD.
$$
Hence using \eqref{EL3.2B} and choosing $g(x)=-f(x, \bar{u}) -\rho\Phi_1+C_1+h^-\geq 0$ for
 $\rho\leq \bar\rho_1$ we have
$$
\Psidel\bar{u}= f(x, \bar{u}) + \rho\Phi + h + g \quad \text{in}\; \cD.
$$
This proves (2).

Now we come to (3). Since $f(x, q)\geq -U_1 q - C$, by Assumption [AP], applying Lemma~\ref{L3.1} we obtain that
\begin{align}\label{EL3.2D}
\hat{u}(x) &= \Exp^x\left[e^{-\int_0^t U_1(X_s)\, \D{s}} \hat{u}(X_t)\Ind_{\{t<\uptau\}}\right] +
 \Exp^x\left[\int_0^{t\wedge\uptau} e^{-\int_0^{s} U_1(X_p)\, \D{p}}(f(x, \hat{u}) + \rho\Phi + h + g+U_1\hat{u})(X_s)\, \D{s}\right]\nonumber
 \\
&\geq \Exp^x\left[e^{-\int_0^t U_1(X_s)\, \D{s}} \hat{u}(X_t)\Ind_{\{t<\uptau\}}\right] +
\Exp^x\left[\int_0^{t\wedge\uptau} e^{-\int_0^{s} U_1(X_p)\, \D{p}}(\rho \Phi +h -C)(X_s)\, \D{s}\right].
\end{align}
Also using \eqref{EL3.2A} and Lemma~\ref{L3.1} we have
\begin{equation}\label{EL3.2E}
\underline{u}(x)= \Exp^x\left[e^{-\int_0^t U_1(X_s)\, \D{s}} \underline{u}(X_t)\Ind_{\{t<\uptau\}}\right]
+\Exp^x\left[\int_0^{t\wedge\uptau} e^{-\int_0^{s} U_1(X_p)\, \D{p}}(\rho \Phi +h -C_1)(X_s)\, \D{s}\right].
\end{equation}
By our choice of $C_1$, we obtain from \eqref{EL3.2D} and \eqref{EL3.2E} that
$$w(x)\geq \Exp^x\left[e^{-\int_0^t U_1(X_s)\, \D{s}} w(X_t)\Ind_{\{t<\uptau\}}\right], \quad t\geq 0,$$
for $w=\hat{u}-\underline{u}$. Since $\lambda^*_{U_1}>0$, we obtain from \cite[Theorem~4.2]{BL17b} that $w\geq 0$ in $\Rd$.
Hence the result.
\end{proof}

Using Lemma~\ref{L3.2} we can now prove the existence of a minimal solution applying monotone iteration scheme.
\begin{lemma}\label{L3.3}
Suppose that the conditions of Lemma~\ref{L3.2} hold.
Then for $\rho\leq \rho_1$, where $\rho_1$ is same value as in Lemma~\ref{L3.2}, there exists $u\in\cC_0(\cD)$ 
satisfying
\begin{equation}\label{EL3.3A}
\Psidel u = f(x, u) + \rho \Phi_1 + h(x) \quad \text{in}\; \cD.
\end{equation}
Moreover, the above $u$ can be chosen to be minimal in the sense that if $\tilde u\in\cC_0(\cD)$ is another 
solution of \eqref{EL3.3A}, then $\tilde u \geq u$ in $\Rd$.
\end{lemma}

\begin{proof}
The proof is based on the standard monotone iteration method. Denote by $m=\min_{\bar\cD} \underline{u}$ and $M=\max_{\bar\cD}\bar{u}$. 
Let $\theta>0$ be a Lipschitz constant for $f(x, \cdot)$ on the interval $[m, M]$, i.e.,
$$
\abs{f(x, q_1)-f(x, q_2)}\leq \theta \abs{q_1-q_2}\quad \text{for}\; q_1, q_2\in [m, M], \; x\in\bar\cD.
$$
Denote $F(x, u)= f(x, u) + \rho\Phi(x)+h(x)$. Consider the solutions of the following family of problems:
\begin{equation}\label{EL3.3B}
\begin{split}
\Psidel u^{(n+1)} + \theta u^{(n+1)} &= F(x, u^{(n)}) + \theta u^{(n)}\quad \text{in}\; \cD,
\\
u^{(n+1)} &= 0 \quad \text{in}\; \cD^c.
\end{split}
\end{equation}
By Theorem~\ref{T2.1},  \eqref{EL3.3B} has a unique solution. We claim that 
\begin{equation}\label{EL3.3C}
\underline{u}=u^{(0)}\leq u^{(n)}\leq u^{(n+1)}\leq \bar{u}\quad \text{for all}\; n\geq 1.
\end{equation}
Denote $w^{(n)}=u^{(n)}-u^{(n-1)}$. Then using Lemma~\ref{L3.1} it is easily seen that
\begin{align}\label{EL3.3D}
w^{(n+1)}(x) &= \Exp^x\left[ e^{-\theta t} w^{(n+1)}(X_t)\Ind_{\{t<\uptau\}}\right]\nonumber
\\
&\quad + \Exp^x\left[\int_0^{t\wedge\uptau} e^{-\theta s}\left(F(X_s, u^{(n)})-F(X_s, u^{(n-1)}) + \theta(u^{(n)}-u^{(n-1)})\right)\, \D{s}\right]
\end{align}
We note that for $n=0$ the right most term in \eqref{EL3.3D} vanishes. Therefore,
$$w^{(1)}(x) \geq \Exp^x\left[ e^{-\theta t} w^{(1)}(X_t)\Ind_{\{t<\uptau\}}\right]\quad t\geq 0.$$
From \cite[Theorem~4.2]{BL17b} we find $w^{(1)}\geq 0$. Note that if $u^{(n)}-u^{(n-1)}\geq 0$ we have 
$$w^{(n+1)}(x) \geq \Exp^x\left[ e^{-\theta t} w^{(n+1)}(X_t)\Ind_{\{t<\uptau\}}\right]\quad t\geq 0,$$
and therefore, we can apply induction to obtain $\underline{u}=u^{(0)}\leq u^{(n)}\leq u^{(n+1)}$. Denoting $v^{n}=\bar{u}-u^{(n)}$ we again
write
\begin{align*}
v^{(n+1)}(x) &\geq \Exp^x\left[ e^{-\theta t} v^{(n+1)}(X_t)\Ind_{\{t<\uptau\}}\right]
\\
&\quad + \Exp^x\left[\int_0^{t\wedge\uptau} e^{-\theta s}\left(F(X_s, \bar{u})-F(X_s, u^{(n)}) + \theta(\bar{u}-u^{(n)})\right)\, \D{s}\right].
\end{align*}
Again employing an induction argument we have $u^{(n)}\leq \bar{u}$. This proves our claim \eqref{EL3.3C}. Therefore, the right hand
side of \eqref{EL3.3B} is bounded uniformly in $n$. Hence by \cite[Theorem~1.1]{KKLL} we obtain
$$\abs{u^{(n)}(x)-u^{(n)}(y)}\leq \kappa\, V(|x-y|)\quad x, y\in\cD, \; n\geq 1.$$
This gives equicontinuity to the family $\{u^{(n)}\}_{n\geq 1}$. Hence by Arzel\`{a}-Ascoli theorem we get that $u^{(n)}\uparrow u$ uniformly in $\Rd$. Thus
we obtain a solution $u$ by passing to the limit in \eqref{EL3.3B}.

To establish minimality we consider a solution $\tilde u$ of \eqref{EL3.3A} in $\cC_0(\cD)$.
From Lemma~\ref{L3.2}(3) we see that $\underline{u}\leq \tilde u$ in $\Rd$. Thus $\bar{u}$ can be replaced by 
$\tilde u$, and the above argument shows that $u\leq \tilde u$.
\end{proof}

Now we derive a priori bounds on the solutions of \eqref{E-AP}. Our first result bounds the negative part of solutions 
$u$ of \eqref{E-AP}.

\begin{lemma}\label{L3.4}
Suppose that Assumption~\ref{A2.1} and [AP](2) hold.
There exists a constant $\kappa = \kappa(d, \Psi, \cD, U_1)$, such that for any solution $u$ of \eqref{E-AP} with $\rho
\geq -\hat\rho, \hat\rho>0,$ we have
$$
\sup_{\cD}\abs{u^-}\;\leq\; \kappa (C+\hat\rho+\norm{h}_\infty),
$$
where $C$ is same constant as in \eqref{AP2}.
\end{lemma}

\begin{proof}
Let $u$ be a solution to \eqref{E-AP} for some $\rho\geq -\hat\rho$. Denote by $w=u\wedge 0$. Then by Lemma~\ref{L3.1} we get
\begin{align*}\label{EL3.4A}
{u}(x) &= \Exp^x\left[e^{-\int_0^t U_1(X_s)\, \D{s}} {u}(X_t)\Ind_{\{t<\uptau\}}\right] +
 \Exp^x\left[\int_0^{t\wedge\uptau} e^{-\int_0^{s} U_1(X_p)\, \D{p}}(f(x, u) + \rho\Phi + h +U_1{u})(X_s)\, \D{s}\right]\nonumber
 \\
&\geq \Exp^x\left[e^{-\int_0^t U_1(X_s)\, \D{s}} {u}(X_t)\Ind_{\{t<\uptau\}}\right] +
\Exp^x\left[\int_0^{t\wedge\uptau} e^{-\int_0^{s} U_1(X_p)\, \D{p}}(\rho \Phi +h -C)(X_s)\, \D{s}\right]\nonumber
\\
&\geq \Exp^x\left[e^{-\int_0^t U_1(X_s)\, \D{s}} w(X_t)\Ind_{\{t<\uptau\}}\right] +
\Exp^x\left[\int_0^{t\wedge\uptau} e^{-\int_0^{s} U_1(X_p)\, \D{p}}(-\hat\rho\Phi_1-\norm{h}_\infty-C)\, \D{s}\right]
\end{align*}
since the right hand side of the above display is non-positive we have 
\begin{equation}\label{EL3.4A}
w(x)\geq \Exp^x\left[e^{-\int_0^t U_1(X_s)\, \D{s}} w(X_t)\Ind_{\{t<\uptau\}}\right] +
\Exp^x\left[\int_0^{t\wedge\uptau} e^{-\int_0^{s} U_1(X_p)\, \D{p}}(-\hat\rho\Phi_1-\norm{h}_\infty-C)\, \D{s}\right],
\end{equation}
for $t\geq 0$ and $x\in\cD$. Let $v\in\cC_0(\cD)$ be the unique solution of 
\begin{equation}\label{EL3.4B}
\Psidel v + U_1 v= -\hat\rho\Phi_1-\norm{h}-C\quad \text{in}\; \cD.
\end{equation}
This is assured by Theorem~\ref{T2.1}. Using Lemma~\ref{L3.1} we see that
$$v(x)= \Exp^x\left[e^{-\int_0^t U_1(X_s)\, \D{s}} v(X_t)\Ind_{\{t<\uptau\}}\right] +
\Exp^x\left[\int_0^{t\wedge\uptau} e^{-\int_0^{s} U_1(X_p)\, \D{p}}(-\hat\rho\Phi_1-\norm{h}_\infty-C)\, \D{s}\right].
$$
Combining with \eqref{EL3.4A} we find
\begin{equation}\label{EL3.4C}
(w-v)(x)\geq \Exp^x\left[e^{-\int_0^t U_1(X_s)\, \D{s}} (w-v)(X_t)\Ind_{\{t<\uptau\}}\right] \quad t\geq 0, \; x\in\cD.
\end{equation}
Since $\lambda^*_{U_1}>0$, using \eqref{EL3.4C} and \cite[Theorem~4.2]{BL17b} we obtain that $w\geq v$ in $\Rd$.
From \eqref{EL3.4B} and \cite[Th.~4.7]{BL17b} we obtain a constant $\kappa = \kappa(d,\Psi, \cD,U_1)$ satisfying
$$
\sup_{x\in\bar\cD}\abs{v}\;\leq\; \kappa (C+\hat\rho + \norm{h}_\infty)
$$
holds. Thus $u^-=-w\leq \kappa (C+\hat\rho + \norm{h}_\infty)$, for $x\in\cD$, and the result follows.
\end{proof}

Our next result provides a lower bound on the growth of the solution for large $\rho$.

\begin{lemma}\label{L3.5}
Let Assumption~\ref{A2.1} and [AP](1)-(2) hold.
For every $\hat\rho>0$ there exists $C_3>0$ such that for every solution $u$ of \eqref{E-AP} with $\rho\geq -\hat\rho$ 
we have
$$
\rho^+\leq C_3(1+\norm{u^+}_\infty) \leq C_3(1+\norm{u}_\infty).
$$
\end{lemma}

\begin{proof}
Let $\varphi=u-\frac{\rho}{\lambda^*}\Phi_1$. Then we have $\varphi\in\cC_0(\cD)$. Also,
\begin{align*}
\Psidel\varphi (x)=f(x, u)+\rho\Phi_1+h-\rho\Phi_1 = f(x, u)-h
\end{align*}
In particular,
$$\varphi(x)=\Exp^x\left[\int_0^\uptau (f(X_s, u(X_s))-h(X_s))\, \D{s}\right], \quad x\in\cD.$$
By our assumption on $f$ and Lemma~\ref{L3.4} we can find a constant $C_4 = C_4(\norm{h}_\infty,\norm{U_1}_\infty, C, \hat\rho)$ satisfying
$$
f(x, u)-h\geq -U_1(x) u -C-\norm{h}_\infty\geq -U_1(x) u^+-\norm{U_1 u^-}_\infty -C-\norm{h}_\infty\geq -C_4(u^+(x)+1).
$$
It then follows that with a constant $C_5$, dependent on $\diam \cD$, 
$$
\sup_{\cD} (-\varphi)^+\leq C_5 C_4 (1+ \norm{u^+}_\infty)
$$
holds. Pick $x\in\cD$ such that $\Phi_1(x)=1$; this is possible since $\norm{\Phi_1}_\infty=1$ by assumption. 
It gives 
$$
\frac{\rho}{\lambda^*}-u(x) \leq (-\varphi(x))^+\leq C_5 C_4 (1+ \norm{u^+}_\infty),
$$
which, in turn, implies
$$
\rho\leq \lambda^* \left( C_4 C_5  +(1+C_4 C_5)\norm{u^+}_\infty\right),
$$
proving the claim. 
\end{proof}

One may notice that we have not used the second condition in \eqref{AP1} so far. The next result makes use of this condition to 
establish an upper bound on the growth of $u$.

\begin{lemma}\label{L3.6}
Suppose that Assumption~\ref{A2.1} and [AP] hold.
For each $\hat\rho>0$ there exists $C_0$ such that for every solution $u$ of \eqref{E-AP}, for $\rho\geq -\hat\rho$, 
we have
\begin{equation}\label{EL3.6A}
\norm{u}_\infty\leq C_0.
\end{equation}
In particular, there exists $\rho_2>0$ such that \eqref{E-AP} does not have any solution for $\rho\geq \rho_2$.
\end{lemma}

\begin{proof}
Suppose, to the contrary, that there exists a sequence $(\rho_n, u_n)_{n\in\mathbb N}$ satisfying \eqref{E-AP} with 
$\rho_n\geq -\hat\rho$ and $\norm{u_n}_\infty\to \infty$. From Lemma~\ref{L3.4} it follows that $\norm{u^+_n}_\infty
=\norm{u_n}_\infty$. Define $v_n=\frac{u_n}{\norm{u_n}_\infty}$. Then 
\begin{equation}\label{EL3.6B}
\Psidel v_n=H_n(x)=\frac{1}{\norm{u_n}_\infty}\left(f(x, u_n) + \rho_n\Phi_1 + h\right)\quad \text{in}\; \cD.
\end{equation}
Since $\norm{H_n}_\infty$ is uniformly bounded by Lemma~\ref{L3.5}, it follows by 
\cite[Theorem~1.1]{KKLL} that
$$
\sup_n\;\norm{v_n}_{\cC^{\bar\phi}(\cD)}\leq \kappa_1,
$$
for some constant $\kappa_1$ and $\norm{\cdot}_{\cC^{\bar\phi}(\cD)}$ is given by \eqref{phi-norm}.
Hence we can extract a subsequence of $\seq v$, denoted by the original sequence, such that it converges to a 
continuous function $v\in\cC_0(\cD)$ in $\cC(\Rd)$. Denote 
\begin{align*}
G_n(x)&= \frac{1}{\norm{u_n}_\infty}(f(x, u_n(x))+h(x)+U_2(x)u_n(x)+ \rho_n\Phi_1(x)),
\\
I_n(x) & = \frac{1}{\norm{u_n}_\infty}(f(x, -u^-_n(x))+h(x)-U_2(x)u^-_n(x)-C +(\rho_n\wedge 0)\Phi_1(x)).
\end{align*}
It then follows from \eqref{AP3} that $G_n\geq I_n$ and $I_n\to 0$ uniformly by Lemma~\ref{L3.5}.
Using \eqref{EL3.6B} and Lemma~\ref{L3.1}, we get
\begin{align}\label{EL3.6C}
v_n(x) &= \Exp^x\left[e^{-\int_0^t U_2(X_s)\, \D{s}} v_n(X_t)\Ind_{\{t<\uptau\}}\right] +
 \Exp^x\left[\int_0^{t\wedge\uptau} e^{-\int_0^{s} U_2(X_p)\, \D{p}}G_n(X_s)\, \D{s}\right]\nonumber
 \\
&\geq \Exp^x\left[e^{-\int_0^t U_2(X_s)\, \D{s}} v_n(X_t)\Ind_{\{t<\uptau\}}\right] +
\Exp^x\left[\int_0^{t\wedge\uptau} e^{-\int_0^{s} U_2(X_p)\, \D{p}}I_n(X_s)\, \D{s}\right].
\end{align}
 Letting $n\to\infty$ in \eqref{EL3.6C} and using the uniform convergence of $I_n$ and $v_n$, we obtain
\begin{equation}\label{EL3.6D}
v(x)\geq \Exp^x\left[e^{-\int_0^t U_2(X_s)\, \D{s}} v(X_t)\Ind_{\{t<\uptau_\cD\}}\right]\quad 
\text{for all}\; x\in\cD, \; t\geq 0.
\end{equation}
Since $\norm{v}_\infty=1$ and $v\geq 0$ in $\Rd$, it is easily seen from \eqref{EL3.6D} that $v>0$ in $\cD$. 
Hence by \cite[Prop.~4.1]{BL17b} it follows that $\lambda^*_{U_2}\geq 0$, contradicting \eqref{AP1}. This 
proves the first part of the result. The second part follows by Lemma~\ref{L3.5} and \eqref{EL3.6A}.
\end{proof}

With the above results in hand, we can now proceed to prove Theorem~\ref{T-AP}. Define
$$
\cA=\big\{\rho\in\RR\; :\; \eqref{E-AP}\; \text{has a solution}\big\}.
$$
By Lemma~\ref{L3.3} we have that $\cA \neq \emptyset$, and Lemma~\ref{L3.6} imply that $\cA$ is bounded
from above. Define $\rho^*=\sup\cA$. Note that if $\rho'<\rho^*$, then $\rho'\in\cA$. Indeed, there is $\tilde{\rho}
\in (\rho', \rho^*)\cap\cA$ and the corresponding solution $u^{(\tilde{\rho})}$ of \eqref{E-AP} with $\rho=\tilde{\rho}$
is a super-solution at level $\rho'$, i.e.,
$$
\Psidel u^{(\tilde{\rho})}=  f(x, u^{(\tilde{\rho})}) + \rho' \Phi_1 + h(x) +g(x) \quad \text{in}\; \cD,
\quad \text{and} \quad u^{(\tilde{\rho})}=0\quad \text{in}\; \cD^c,
$$
where $g(x)=(\tilde\rho-\rho')\Phi_1\geq 0$.
Using Lemma~\ref{L3.2}(3) and from the proof of Lemma~\ref{L3.3} we have a minimal solution of \eqref{E-AP} with $\rho=\rho'$. Next we show
that there are at least two solutions for $\rho<\rho^*$.

Recall that $\delta_\cD:\bar\cD\to [0, \infty)$ is the distance function from the set $\cD^c$. We can assume that
$\delta_\cD$ is a positive $\cC^1$-function in $\cD$. For a sufficiently small $\varepsilon>0$, to be chosen later, consider
the Banach space
$$
\mathfrak X=\left\{\psi\in \cC_0(\cD)\; :\; \left\|{\frac{\psi}{V(\delta_\cD)}}\right\|_{\cC^\varepsilon(\cD)} < \infty\right\}.
$$
In fact, it is sufficient to consider any $\varepsilon$ strictly smaller than the parameter $\alpha$
in \cite[Th.~1.2]{KKLL}. It should be observed that for every $\psi\in\mathfrak{X}$ we can extend $\psi\cdot [V(\delta_\cD)]^{-1}$ up to the boundary
$\partial\cD$ continuously.

For $\rho\in\RR$ and $m\geq 0$ we define a map $K_\rho:\mathfrak X\to\mathfrak X$ as follows. For $v\in\mathfrak X$,
$K_\rho v =u$ is the unique solution (see Theorem~\ref{T2.1}) to the Dirichlet problem
$$
\Psidel u + m u = f(x, v) + \rho \Phi_1 + h(x)+ mv \quad \text{in}\; \cD, \quad \text{and} \quad u=0\quad \text{in}\; \cD^c.
$$
It follows from \cite[Th.~1.2]{KKLL}
$$\left\|{\frac{\psi}{V(\delta_\cD)}}\right\|_{\cC^\alpha(\cD)}<\infty,$$
for $\alpha>\varepsilon$, and thus $u\in\mathfrak X$. In fact, using the above estimate it can be easily shown that $K_\rho$ is continuous and compact.

\begin{lemma}\label{L3.7}
Let $\rho<\rho^*$. Then there exists $m\geq 0$ and an open $\cO \subset \mathfrak X$, containing the minimal solution,
satisfying $\Deg(I-K_\rho, \cO, 0)=1$.
\end{lemma}

\begin{proof}
We borrow some of the arguments of \cite{DF80}(see also \cite{BL18}) with a suitable modification. Pick $\bar{\rho}\in (\rho, \rho^*)$ and
let $\bar{u}$ be a solution of \eqref{E-AP} with $\rho=\bar{\rho}$. It then follows that
$$
\Psidel \bar{u} = f(x, \bar{u}) + \rho \Phi_1 + h(x) + g(x) \quad \text{in}\; \cD \quad \text{and} \quad u=0\quad \text{in}\; \cD^c,
$$
for $\bar{g}(x)=(\bar\rho-\rho)\Phi_1$ and by Lemma~\ref{L3.3}(1) we have a classical subsolution
$$
\Psidel \underline{u} = f(x, \underline{u}) + \rho \Phi_1 + h(x) + \underline{g}(x) \quad \text{in}\; \cD \quad \text{and} \quad u=0\quad
\text{in}\; \cD^c,
$$
with $\underline{g}\leq 0$.
Then Lemma~\ref{L3.2}(3) supplies $\underline{u}\leq\bar{u}$ in $\Rd$, hence the minimal solution $u$ of
\eqref{E-AP} satisfies $\underline{u}\leq u\leq\bar{u}$ in $\Rd$. Note that for every $\psi\in\mathfrak X$, the ratio
$\frac{\psi}{V(\delta_\cD)}$ is continuous up to the boundary. Define
$$
\cO=\left\{\psi\in\mathfrak X\; :\; \underline{u}<\psi<\bar{u}\; \text{in}\; \cD, \;\; \frac{\underline{u}}{V(\delta_\cD)}<
\frac{\psi}{V(\delta_\cD)} < \frac{\bar u}{V(\delta_\cD)}\; \text{on}\; \partial\cD, \;\; \norm{\psi}_{\mathfrak X}< r\right\},
$$
where the value of $r$ will be chosen later. It is clear that $\cO$ is bounded, open and convex. Also,
if we choose $r$ large enough, then the minimal solution $u$ belongs to $\cO$. Indeed, note that for $w=u-\underline{u}$
\begin{equation}\label{EL3.7A}
\Psidel w=f(x, u)-f(x,\underline{u}) - \underline{g}\quad \text{in}\; \cD.
\end{equation}
Define 
\[U(x)=\left\{\begin{array}{lll}
\left(\frac{f(x, \underline{u}(x))-f(x, u(x))}{\underline{u}(x)-u(x)}\right)^- & \text{if}\; u(x)\neq\underline{u}(x)\,,
\\
\left(D_uf(u(x), x)\right)^- & \text{if}\; u(x)=\underline{u}(x)\,.
\end{array}
\right.
\]
By Assumption~[AP](1) we have $U\in\cC(\bar\cD)$. Also note that
\begin{align*}
f(x, u)-f(x,\underline{u}) + U(x) w=\left(\frac{f(x, \underline{u}(x))-f(x, u(x))}{\underline{u}(x)-u(x)}\right)^+ w\geq 0,
\end{align*}
since $w\geq 0$. Now applying Lemma~\ref{L3.1} to \eqref{EL3.7A} we obtain that
\begin{equation}\label{EL3.7B}
w(x)\geq \Exp^x\left[e^{-\int_0^t U_1(X_s)\, \D{s}} w(X_t)\Ind_{\{t<\uptau\}}\right],\quad t\geq 0.
\end{equation}
Using estimate \eqref{EL3.2C1} it is obvious that $w>0$ in $\cD$. Choose $t=2$, $\cD_1\Subset\cD$ and use \eqref{EL3.7B} to obatin
\begin{align*}
w(x) &= e^{-2\norm{U}_\infty} \Exp^x\left[ w(X_t)\Ind_{\{2<\uptau\}}\right]
\\
&= e^{-2\norm{U}_\infty} \int_\cD w(y) p^\cD(2, x, y)\, \D{y}
\\
&\geq e^{-2\norm{U}_\infty} \int_{\cD_1} w(y) p^\cD(2, x, y)\, \D{y}
\\
&\geq \kappa_1 e^{-2\norm{U}_\infty}\, \min_{\cD_1} w\, \Prob^x(\uptau>1) \int_{\cD_1} \Prob^y(\uptau>1) p(1\wedge V^2(r), \abs{x-y})\, \D{y}
\\
&\geq \kappa_2\, p(1\wedge V^2(r), 0) \Prob^x(\uptau>1) \int_{\cD_1} \Prob^y(\uptau>1)\, \D{y},
\end{align*}
for some constant $\kappa_2$, where in the fourth inequality we use \eqref{EL3.2C1}. Now using \eqref{EL3.2C2} we can find a constant $\kappa_3>0$ satisfying
$$
w(x)\geq \kappa_3 V(\delta_\cD(x)), \quad x\in \cD.
$$
This of course, implies
$$
\min_{\partial\cD}\left(\frac{u}{V(\delta_\cD)}-\frac{\underline{u}}{V(\delta_\cD)}\right)>0.
$$
Similarly, we can compare also $u$ and $\bar{u}$.

We define $m$ to be a Lipschitz constant of $f(x, \cdot)$ in the interval $[\min\underline{u}, \max\bar{u}]$. Also, define
\[
\tilde{f}(x, q)= f \left(x, (\underline{u}(x)\vee q)\wedge\bar{u}(x)\right) + m (\underline{u}(x)\vee q)\wedge\bar{u}(x).
\]
Note that $f$ is bounded and Lipschitz continuous in $q$, and also non-decreasing in $q$. We define another map $\tilde
K_\rho: \mathfrak X \to \mathfrak X$ as follows: for $v\in\mathfrak X$, $\tilde K_\rho v=\tilde{u}$ is the unique 
solution of
\begin{equation}\label{EL3.7C}
\Psidel \tilde{u} + m \tilde{u} = \tilde{f}(x, v) + \rho\Phi + h  \quad \text{in}\; \cD, \quad \text{and} \quad u=0\quad \text{in}\; \cD^c.
\end{equation}
It is easy to check that $K_\rho$ is a compact mapping. Since the right hand side of \eqref{EL3.7C} is bounded,
using again \cite[Th.~1.2]{KKLL}, we find $r$ satisfying
$$
\sup\left\{\norm{\tilde K_\rho v}_{\mathfrak X}\; :\; v\in\mathfrak X\right\} < r.
$$
We fix this choice of $r$. We now show that $\tilde{K}_\rho v\in\cO$ for all $v\in\mathfrak{X}$. Let $\tilde{u}=\tilde{K}_\rho v$.
Then
\begin{align*}
\Psidel(u-\underline{u})  = -m(u-\underline{u}) + \tilde{f}(x, v) -m\underline{u} - f(x, \underline{u})-\underline{g}
\end{align*}
Since 
$$
\tilde{f}(x, v) -m\underline{u} - f(x, \underline{u})-\underline{g}\geq  \tilde{f}(x, \underline{u}) -m\underline{u} - f(x, \underline{u})=0,
$$
from Lemma~\ref{L3.1} we note that for $w=\tilde{u}-\underline{u}$
\begin{equation}\label{EL3.7D}
w(x)\geq \Exp^x\left[e^{- mt} w(X_t)\Ind_{\{t<\uptau\}}\right]\,.
\end{equation}
Thus letting $t\to\infty$ in \eqref{EL3.7D} we have obtain $w\geq 0$. Since $\underline{g}\lneq 0$, it follows from \eqref{EL3.7C} that $w$ can not be identically $0$.
Hence again applying \eqref{EL3.7D} we obtain $w>0$ in $\cD$. Repeating the arguments as above (see below \eqref{EL3.7B}) we also have
$$
\min_{\partial\cD}\left(\frac{\tilde{u}}{V(\delta_\cD)}-\frac{\underline{u}}{V(\delta_\cD)}\right)>0.
$$
The other estimates with respect to $\bar{u}$ can be obtained similarly. Finally, this implies that $\tilde K_\rho v\in \cO$, for all $v\in\mathfrak X$.
Moreover, $0\notin (I-\tilde K_\rho)(\partial\cD)$. Then by the homotopy invariance property of degree we find that
$\Deg(I-\tilde K_\rho,\cO, 0)=1$ (see for instance, \cite{DF80}). Since $\tilde K_\rho$ coincides with $K_\rho$ in $\cO$, we obtain $\Deg(I-K_\rho, \cO, 0)=1$.
\end{proof}

Similarly as before, define $\cS_\rho:\mathfrak X\to\mathfrak X$ such that for $v\in\mathfrak X$, $u=\cS_\rho v$ is given
by the unique solution of
$$
\Psidel u = f(x, v) + \rho \Phi_1 + h(x) \quad \text{in}\; \cD, \quad \text{and} \quad u=0\quad \text{in}\; \cD^c.
$$
Then the standard homotopy invariance of degree (w.r.t. $m$) gives that $\Deg(I-\cS_t, \cO, 0)=1$. This observation will be helpful
in concluding the proof below.

\begin{proof}[Proof of Theorem~\ref{T-AP}]
Using Lemma~\ref{L3.7} we can now complete the proof by using \cite{DF80,DFS}. Recall the map $\cS_\rho$ defined above, and
fix $\rho<\rho^*$. Denote by $\cO_R$ a ball of radius $R$ in $\mathfrak X$. From Lemma~\ref{L3.6} and \cite[Theorem~1.2]{KKLL}  we find that
$$
\Deg(I-\cS_{\tilde \rho}, \cO_R, 0)=0\quad \text{for all}\;\, R>0, \; \tilde{\rho}\geq \rho_2.
$$
Using again Lemmas~\ref{L3.6} and \cite[Th.~1.2]{KKLL}, we obtain that for every $\hat\rho$ there exists a constant
$R$ such that
$$
\norm{u}_{\mathfrak X} < R
$$
for each solution $u$ of \eqref{E-AP} with $\tilde{\rho}\geq -\hat\rho$. Fixing $\hat\rho>\abs{\rho}$ and the corresponding
choice of $R$, it then follows from homotopy invariance that $\Deg(I-\cS_\rho, \cO_R, 0)=0$. We can choose $R$ large enough so
that $\cO\subset\cO_R$ where $\cO$ is from Lemma~\ref{L3.7}.
 Since $\Deg(I-\cS_{\rho}, \cO, 0)=1$, as seen above, using the excision property of degree we conclude that
there exists a solution of \eqref{E-AP} in $\cO_R\setminus\cO$. Hence for every $\rho<\rho^*$ there exist at least two
solutions of \eqref{E-AP}. The existence of a solution at $\rho=\rho^*$ follows from the a priori estimates in Lemma~\ref{L3.6},
the estimate in \cite[Theorem.~1.1]{KKLL}, and the stability property of the semigroup solutions. This completes
the proof of Theorem~\ref{T-AP}.
\end{proof}

\subsection*{Acknowledgments}
This research of Anup Biswas was supported in part by an INSPIRE faculty fellowship and a DST-SERB grant EMR/2016/004810.

\end{document}